\documentclass[11pt]{amsart}
\usepackage{graphicx}
\usepackage[active]{srcltx}
 \makeatletter
\renewcommand*\subjclass[2][2000]{%
  \def\@subjclass{#2}%
  \@ifundefined{subjclassname@#1}{%
    \ClassWarning{\@classname}{Unknown edition (#1) of Mathematics
      Subject Classification; using '1991'.}%
  }{%
    \@xp\let\@xp\subjclassname\csname subjclassname@#1\endcsname
  }%
}
 \makeatother
\usepackage{enumerate,amssymb,  mathrsfs,yhmath}%, pdfsync}% ,showkeys, pdfsync}

\newtheorem{theorem}{Theorem}[section]
\newtheorem{lemma}[theorem]{Lemma}
\newtheorem*{lemma*}{Lemma}

\def\1ton{1,2,\ldots,n}

\usepackage{amssymb}%, pdfsync}
\usepackage{amsthm}
\usepackage{mathrsfs, amsfonts, amsmath}
\usepackage{graphicx}

\theoremstyle{definition}
\newtheorem{definition}[theorem]{Definition}

\theoremstyle{remark}
\newtheorem{remark}[theorem]{Remark}

\numberwithin{equation}{section}

\newcommand{\abs}[1]{\lvert#1\rvert}

\def\XXint#1#2#3{{\setbox0=\hbox{$#1{#2#3}{\int}$}
\vcenter{\hbox{$#2#3$}}\kern-.5\wd0}}

\def\ge{\geqslant}
\begin{document}

\title[On Riesz conjugate function theorem for harmonic functions]{ On M. Riesz conjugate function theorem for harmonic functions}

%\author{Josip Globevnik}
%\address{Institute of Mathematics, Physics and Mechanics
%Jadranska 19, 1000 Ljubljana, Slovenia}
%\email{josip.globevnik@fmf.uni-lj.si}

\author{David Kalaj}
\address{Faculty of Natural Sciences and Mathematics, University of
Montenegro, Cetinjski put b.b. 81000 Podgorica, Montenegro}
\email{davidk@ucg.ac.me}

%\def\thefootnote{}
%\footnotetext{ \texttt{\tiny File:~\jobname.tex,
 %         printed: \number\year-\number\month-\number\day,
  %        \thehours.\ifnum\theminutes<10{0}\fi\theminutes }
%} \makeatletter\def\thefootnote{\@arabic\c@footnote}\makeatother
%===============================================================================

\begin{abstract}
Let  $L^p(\mathbf{T})$ be the Lesbegue space of complex-valued functions defined in the unit circle $\mathbf{T}=\{z: |z|=1\}\subseteq \mathbb{C}$.
In this paper, we address the problem of finding the best constant
in  the inequality of the form:
$$\|f\|_{L^p(\mathbf{T})}\le A_{p,b}  \|(|P_+ f|^2+b| P_{-} f|^2)^{1/2}\|_{L^p(\mathbf{T})}.$$
Here $p\in[1,2]$, $b>0$, and by $P_{-} f$ and $ P_+ f$ are denoted co-analytic and analytic projection of a function $f\in L^p(\mathbf{T})$. The equality is "attained" for  a quasiconformal harmonic mapping.
The result extends a sharp version of  M. Riesz conjugate function theorem of Pichorides and Verbitsky and some well-known estimates for holomorphic functions.
\end{abstract}
\maketitle
\tableofcontents

\section{Introduction}
Let $\mathbf{U}$ denote the unit disk and $\mathbf{T}$ the unit circle in the complex plane.
For $p>1$, we define the Hardy class $\mathbf{h}^p$ as the class of harmonic mappings $f=g+\bar h$, where $g$ and $h$ are holomorphic mappings defined on the unit disk $\mathbf{U},$ so that $$\|f\|_p=\|f\|_{\mathbf{h}^p}=\sup_{0<r< 1} M_p(f,r)<\infty,$$ where $$M_p(f,r)=\left(\int_{\mathbf{T}}|f(r\zeta)|^p d\sigma(\zeta)\right)^{1/p}.$$ Here $d\sigma(\zeta)=\frac{dt}{2\pi},$ if $\zeta=e^{it}\in \mathbf{T}$. The subclass of holomorphic mappings that belongs to the class $\mathbf{h}^p$ is denoted by $H^p$.

If $f\in \mathbf{h}^p$, then it is well-known that there exists
$$f(e^{it})=\lim_{r\to 1} f(re^{it}),   a.e.$$  and $f\in L^p(\mathbf{T}).$
Then there hold
\begin{equation}\label{come}\|f\|^p_{{\mathbf{h}^p}}=\lim_{r\to 1}\int_{0}^{2\pi}|f(re^{it})|^p \frac{dt}{2\pi}= \int_{0}^{2\pi}|f(e^{it})|^p \frac{dt}{2\pi}.\end{equation}

%Similarly we define the {\it Hardy space} $H^p$ of holomorphic
%functions.

Let $1<p<\infty$ and let $\overline p=\max\{p,p/(p-1)\}$. Verbitsky in \cite{ver} proved the following results.  If $f=u+iv\in H^p$ and $v(0)=0$, then  \begin{equation}\label{ver}\sec(\pi/(2\overline p))\|v\|_p\leq\|f\|_p,\end{equation}
and
\begin{equation}\label{1ver}\|f\|_p\leq\csc(\pi/(2\overline p))\|u\|_p,\end{equation}
and both estimates are sharp. Those results improve the sharp inequality \begin{equation}\label{pico}\|v\|_p\leq\cot(\pi/(2\overline p))\|u\|_p\end{equation} found by S.  Pichorides (\cite{pik}). For some related results see \cite{essen, studia, graf, verb2}.  Then those results have been extended by the author in \cite{tams} by proving the sharp inequalities \begin{equation}\label{nes2}\left(\int_{\mathbf{T}}\left({|g|^2+|h|^2}\right)^{\frac{p}{2}}\right)^{1/p}\le c_p\left(\int_{\mathbf{T}} |g+\bar h|^p\right)^{1/p}\end{equation}
and
\begin{equation}\label{nes3}\left(\int_{\mathbf{T}} |g+\bar h|^p\right)^{1/p}\le d_p\left(\int_{\mathbf{T}}\left({|g|^2+|h|^2}\right)^{\frac{p}{2}}\right)^{1/p},\end{equation} where
$$c_p=\left(\sqrt{2}\sin\frac{\pi}{2\bar p}\right)^{-1},\text{ and }d_p= \sqrt{2}\cos\frac{\pi}{2\bar p},$$ and $f=g+\bar h\in \mathbf{h}^p$, $\Re(h(0)g(0))=0$, $1<p<\infty$. Then inequalities \eqref{nes2} and \eqref{nes3} imply \eqref{pico}, \eqref{1ver} and \eqref{ver}. As a byproduct, the author by using \eqref{nes2} proved a Hollenbeck-Verbitsky conjecture for the case $s=2$.

Further, the inequality \eqref{nes2} has been extended by Markovi\'c and Melentijevi\'c in \cite{melmar} by finding the best constant $c_{p,s}$ in the following inequality \begin{equation}\label{nes4}\left(\int_{\mathbf{T}}(|g|^s+|h|^s)^{p/s}\right)^{1/p}\le c_{p,s}\left(\int_{\mathbf{T}} |g+\bar h|^p\right)^{1/p},\end{equation}  for certain range of parameters $(p,s)$ including the case $(p,2)$. Then \eqref{nes4} has been extended by Melentijevi\'c in \cite{mel}, proving in this way a Hollenbeck-Verbitsky conjecture for the case $s<\sec^2(\pi/(2p))$, $p\le 4/3$ or $p\ge 2$. So this problem remains open for the case $p\in[4/3,2]$.

We will consider another related problem. We will extend the result of Verbitsky \eqref{ver} and the result of the author \eqref{nes3} (for the case $1\le p\le 2$). For a given $b>0$ and $p\in[1,2]$ we will find the best constant in the following inequality  \begin{equation}\label{nes32}\left(\int_{\mathbf{T}} |g+\bar h|^p\right)^{1/p}\le A_{p,b}\left(\int_{\mathbf{T}}\left({|g|^2+b|h|^2}\right)^{\frac{p}{2}}\right)^{1/p}.\end{equation}
The previous estimate is equivalent to the following estimate:
$$\|f\|_{p}\le A_{p,b}\|(| P_+[f]|^2+|bP_{-}[f]|^2)^{\frac{p}{2}}\|_p,$$ where $f\in L^p(\mathbf{T})$ and $P_{+}$ and $P_{-}$ are analytic and co-analytic projection of a function $$f(e^{it})=\sum_{k=-\infty}^{+\infty} c_k e^{ik t},$$ i.e. $$P_{+}[f]=\sum_{k=0}^{+\infty} c_k e^{ik t}, \text{and} \ \  P_{-}[f]=\sum_{k=1}^{+\infty} c_{-k} e^{-ik t}.$$ Here $$c_k = \frac{1}{2\pi}\int_0^{2\pi} f(e^{it})e^{-ki t}dt.$$ We remark that all the inequalities \eqref{ver}, \eqref{pico}, \eqref{nes2}, \eqref{nes3}, \eqref{nes4} can be formulated by using analytic and co-analytic projection.
\section{Main results}
The main result is the following theorem

\begin{theorem}\label{Tao}
Let $1\le p\le 2$ and assume that $b>0$.
%(or more general with $\abs{ \mathrm{arg}\left[g(0)h(0)\right]}\le \pi -\frac{\pi }{p}$)
Then we have the following sharp inequality
\begin{equation}\label{nesit}\left(\int_{\mathbf{T}} |g+\bar h|^p\right)^{1/p}\le A_{p,b}\left(\int_{\mathbf{T}}(|g|^2+b  |h|^2)^{\frac{p}{2}}\right)^{1/p},\end{equation}
for $f=g+\bar h\in \mathbf{h}^p$ with $\Re(g(0)h(0))\le  0$ where
$$A_{p,b}=\left(\frac{1+b+\sqrt{1+b^2+2 b \cos \left[\frac{2 \pi }{p}\right]}}{2 b}\right)^{1/2}.$$
The equality is "attained" for $f_c$, when $c\uparrow 1/p$ and

\begin{equation}\label{fbb}f_c(z) =  \left(\frac{1+z}{1-z}\right)^c -\bar r\left(\frac{1+\bar z}{1-\bar z}\right)^c, \end{equation} and
$$\bar r=\frac{\left(-1+b-\sqrt{1+b^2+2 b \cos\left[\frac{2 \pi }{p}\right]}\right) \sec\left(\frac{\pi }{p}\right)}{2 b}, $$ for $p<2$. For $p=2$   \begin{equation}\label{fbb0}f_c(z) = \left\{
                                                            \begin{array}{ll}
                                                              \left(\frac{1+\bar z}{1-\bar z}\right)^c, & \hbox{for $b\le 1$;} \\
                                                              \left(\frac{1+z}{1-z}\right)^c, & \hbox{for $b>1$.}
                                                            \end{array}
                                                          \right.
 \end{equation}
\end{theorem}

\begin{remark}
a) Theorem~\ref{Tao} remains true if the condition  $\Re(g(0)h(0))\le 0$ is replaced by more general condition $$\frac{\pi(p-1)}{p}\le |\mathrm{arg}(g(0)h(0))|\le \pi.$$ In the previous inequality we considered the branch of $\arg$ with  $\arg(x)\in (-\pi,\pi]$.

b) Theorem~\ref{Tao} reduces to the inequality \eqref{ver} by Verbitsky for $b=1$ and $g=-h$. Namely $h-\bar h=2i\mathrm{Im}\, h$ and $\mathrm{Im}(h(0))=0$ implies $\Re(-h^2(0))\le 0$.   Theorem~\ref{Tao} coincides with the author's inequality \eqref{nes3} for $b=1$ and for the case $p\in[1,2]$. In the last case  $D_{p,b}=d_p$.

c) Concerning the analytic version of the Riesz type theorem,  the following example suggested by A. Calderon (see \cite{pik}) shows that \eqref{ver} and \eqref{1ver} are sharp. Namely if $g_c(z)=\left(\frac{1+z}{1-z}\right)^{c}$, $\abs{\arg\frac{1+z}{1-z}}\le \frac{\pi}{2}$, and $c<\frac{1}{p}$, then $g_c=u +iv\in h^p$. Further $|g_c|=\csc \frac{c\pi}{2} |v|$ almost everywhere on $\mathbf{T}$. In contrast to the analytic case where the  univalent (conformal injective) function  $g$ of the unit disk of a certain angle is the minimizer, for harmonic version of this theorem, the minimizer  $f_b$  is a $\min\{\bar r ,\bar r^{-1}\} -$quasiconformal harmonic mapping of the unit disk onto an angle, provided that $b\neq 1$. The case $b=1$ implies $\bar r=1$, and then $f_b$ from \eqref{fbb} coincides with $2i\mathrm{Im}( g_c)$.
\end{remark}

\section{Strategy of the proof}
As the authors of the paper did in \cite{verb1}, (see also \cite{tams, melmar, mel}) we use "pluri-subharmonic minorant".
\begin{definition}
An upper semi-continuous real function $u$ is called subharmonic in an open set $\Omega$ of the complex plane, if for every compact subset $K$ of $\Omega$ and every harmonic function $f$ defined on $K$, the inequality $u(z)\le f(z)$ for $z\in \partial K$ implies that $u(z)\le f(z)$ on $K$.
\end{definition}
A property which characterizes the subharmonic mappings is the sub-mean value property which states that if $u$ is a subharmonic function defined on a domain $\Omega$, then for every closed disk $\overline{D(z_0,r)}\subset \Omega$, we have the inequality $$u(z_0)\le \frac{1}{2\pi r}\int_{|z-z_0|=r} {u(z)|dz|}.$$
\begin{definition} \cite{lars} A function $u$ defined in an open set $\Omega\subset \mathbf{C}^n$ with values
in $[ - \infty, +\infty)$ is called plurisubharmonic if
\begin{enumerate}
\item $u$ is semicontinuous from above;
\item For arbitrary  $z,w \in\mathbf{C}^n$, the function $t\to  u(z + tw)$ is
subharmonic in the part $\mathbf{C}$ where it is defined.
\end{enumerate}

\end{definition}

\begin{definition} A pluri-subharmonic function $f$ is called a pluri-subharmonic minorant of $g$ on $\Omega$ if $f(z,w)\le g(z,w)$ for $(z,w)\in \Omega\subset\mathbf{C}^2$, and $f(z_0,w_0)=g(z_0,w_0)$, for a point $(z_0,w_0)\in \Omega.$\end{definition}

Let $p\in[1,2]$ and $b>0$.  The main task in the proof of Theorem~\ref{Tao} is to find optimal positive constants $D_{p,b}$, $E_{p,b}$,  and pluri-subharmonic functions $\mathcal{G}_p(z,w)$ (Lemma~\ref{nice1}) for $z,w\in \mathbf{C}$, vanishing for $z=0$ or $w=0$, so that the inequality
\begin{equation}\label{calf}
 |w+\bar z|^p\le (|w|^2+ b |z|^2)^{\frac{p}{2}} D_{p,b} - E_{p,b} \mathcal{G}_p(z,w)
\end{equation}
 is sharp.
\section{Proof of Theorem~\ref{Tao}}
The main content of the proof is the following lemma which we prove in the next section.

Let $\zeta=\rho e^{it}$, $t\in(-\pi,\pi]$, and consider the following branches \begin{equation}\label{br1}(-{\zeta})^{\frac{p}{2}}:=|\zeta|^{p/2}e^{ip/2(\pi+t)}\end{equation}
\begin{equation}\label{br2}(-{\bar \zeta})^{\frac{p}{2}}:=|\zeta|^{p/2}e^{ip/2(\pi-t)}.\end{equation}
Then for $\zeta=\rho e^{it}$, $t\in[-\pi,\pi]$ consider
\begin{equation}\label{rez}k(\rho,t):=\max\{\Re((-{\zeta})^{\frac{p}{2}}),\Re((-\bar{\zeta})^{\frac{p}{2}})\}=\rho^{\frac{p}{2}}\cos(\frac{p}{2}(\pi-|t|)).\end{equation}
%Notice that the second equality in \eqref{rez} is true for $|t|\le \pi/2+\pi/p$.
Further define $$K(\rho,t)=\left\{
                            \begin{array}{ll}
                              k(\rho,t), & \hbox{if $|t|\le \pi$;} \\
                              k(\rho,t-2\pi), & \hbox{if $t\in [\pi,2\pi]$;} \\
                              k(\rho, t+2\pi), & \hbox{if $t\in[-2\pi,-\pi]$.}
                            \end{array}
                          \right.$$

Then for $\zeta=\rho e^{it}\in \mathbb{C}$  define the function $H(\zeta)=K(\rho,t)$. This function is well-defined and continuous in $\mathbb{C}$ because $K(\rho,\pi)=K(\rho,-\pi)$.

\begin{lemma}\label{nice1}
The function  $H(\zeta)$ is subharmonic on $\mathbf{C}$. For $w=|w| e^{it}$, $t\in(-\pi,\pi]$ and $z=|z|e^{is}$, $s\in(-\pi,\pi]$,  define $\mathcal{G}_p(z,w)=K(|z|\cdot |w|, t+s)$.
Then $\mathcal{G}_p$ is pluri-subharmonic on $\mathbf{C}^2.$
For every two complex numbers $z$ and $w$ and $p\in[1,2)$, we have
\begin{equation}\label{fsh} |w+\bar z|^p\le D_{p,b}(|w|^2+b|z|^2)^{\frac{p}{2}} - E_{p,b} \mathcal{G}_p(z,w),\end{equation}

where
\begin{equation}\label{aba}D_{p,b}=\left(\frac{1+b+\sqrt{S_b}}{2 b}\right)^{\frac{p}{2}},\end{equation}
and \begin{equation}\label{abb} E_{p,b}=2^{2-\frac{p}{2}} \left(-\frac{\left(S_b+(1+b) \sqrt{S_b}\right) \sec\left(\frac{\pi }{p}\right)}{b}\right)^{\frac{1}{2} (-2+p)} \sin \frac{\pi}{p}.\end{equation} Here we use the shorthand notation \begin{equation}\label{Sb}S_b:=1+b^2+2 b \cos \left[\frac{2 \pi }{p}\right].\end{equation}
Furthermore the equality in \eqref{fsh} is attained for $p<2$ if and if $${|w|}=-\frac{1}{2} \left(-1+b+\sqrt{S_b}\right) \sec\left(\frac{\pi }{p}\right){|z|}$$ and $\arg(wz)=\frac{-\pi}{p}\mod \pi$.
\end{lemma}

\begin{proof}[Proof of inequality of Theorem~\ref{Tao}]

We use  Lemma~\ref{nice1}. Assume that $f=g+\bar h$, where $g$ and $h$ are holomorphic functions on the unit disk. Then from Lemma~\ref{nice1}, we have $$ |g(z)+\overline{h(z)}|^p \le D_{p,b} (|g(z)|^2+b |h(z)|^2)^{\frac{p}{2}} - E_{p,b} \mathcal{G}_p(g(z),h(z)),$$ where
$D_{p,b}$ and $E_{p,b}$ are given in \eqref{aba} and \eqref{abb}.
Then $$\int_{\mathbf{T}} |g(z)+\overline{h(z)}|^p\le D_{p,b} \int_{\mathbf{T}} (|g(z)|^2+b |h(z)|^2)^{\frac{p}{2}} - E_{p,b} \int_{\mathbf{T}}\mathcal{G}_p(g(z),h(z)).$$
Let $\theta=\arg(g(0)h(0))$. As $\mathcal{P}_p(z)= \mathcal{G}_p(g(z),h(z))$ is subharmonic for every $p\ge 1$, by sub-mean inequality we have that $$\int_{\mathbf{T}}\mathcal{G}_p(g(z),h(z))\ge \mathcal{G}_p(g(0),h(0))=|g(0)h(0)|^p\cos\left[p\frac{\pi-|\theta|}{2}\right]\ge 0,$$ because $\theta\in[\pi/2,3\pi/2]$ and $p\in[1,2]$.
\end{proof}
\begin{proof}[Proof of sharpness of  Theorem~\ref{Tao}]

Let $\beta \in[0,1]$ and define $$ f_\beta=g+\bar h=\beta \left(\frac{1+z}{1-z}\right)^c+(\beta-1)\left(\frac{1+\bar z}{1-\bar z}\right)^c,\, c<1/p.$$ Notice that $\mathrm{arg}(g(0)h(0))=\pi$, or $g(0)h(0)=0$ and thus  $\mathrm{Re}(g(0)h(0))\le 0$. Let  $$Z(\beta):={\left(\int_{\mathbf{T}} |g+\bar h|^p\right)}.$$ Then $$Z(\beta) =2 \pi  (1+2 (-1+\beta) \beta+2 (-1+\beta) \beta \cos[c \pi ])^{\frac{p}{2}} \sec\left[\frac{c p \pi }{2}\right].$$
Further let $$X(\beta) :={\left(\int_{\mathbf{T}} (|g|^2+b |h|^2)^{\frac{p}{2}}\right)}.$$
Then  $$X(\beta)= (1+b)^{\frac{p}{2}}\int_0^{2\pi} \left|\frac{1+e^{it}}{1-e^{it}}\right|^{cp} dt= 2 \pi(1+b)^{\frac{p}{2}}   \sec\left[\frac{c p \pi }{2}\right].$$
Now for $c=1/p$ and for $$\beta=\frac{\left(-1+b+2 b \cos \frac{\pi}{p}+\sqrt{S_b}\right) \sec\left[\frac{\pi }{2 p}\right]^2}{4 (-1+b)},$$ we have $$K(\beta,b,c)=\frac{Z(\beta)}{X(\beta)}={D(p,b)}.$$

We need also to check that $\beta\in [0,1]$. But after strighforward calculation this is equivalent with the inequality $|b-1|\cos\frac{\pi}{p}<0$ which is true for $b\ne 1$ and $p\neq 2$. For $b=1$ we obtain $\beta = 1/2$. If $b\neq 1$, then $\kappa=\min\{\beta/(1-\beta),(1-\beta)/\beta\} \in(0,1)$ and $f$ is quasiconformal harmonic mapping of the unit disk onto an angle.

\end{proof}
\section{Proof of Lemma~\ref{nice1}}
Prove first the subharmonicity of $H$. Let $z=re^{i\theta},$ $\theta\in(-\pi,\pi]$. Notice that \begin{equation}\label{form}\max\{r^{p/2} \cos[p/2(\pi -\theta)],r^{p/2} \cos[p/2(\pi+\theta)]\}=r^{p/2}\cos(\frac{p}{2}(\pi-\abs{\theta}).\end{equation} Thus $$H(z)=r^{p/2}\cos(\frac{p}{2}(\pi-\abs{\theta})$$ is subharmonic as a maximum of harmonic functions \begin{equation}\label{br3} \Re((-{z})^{\frac{p}{2}}), \ \ \Re((-\bar{z})^{\frac{p}{2}})\end{equation} for $z\in \mathbb{C}\setminus (-\infty,0]$. But the same formula \eqref{form} hold for $\theta\in (-\pi/p-\pi/2, \pi/p+\pi/2)\supseteq [-\pi,\pi]$, for $1\le p<2$, and thus $H$ is subharmonic as the maximum of the branches of the same  harmonic functions \eqref{br3}  defined in $\mathrm{Re}(z)\le 0$. Thus $H$ is subharmonic in $\mathbb{C}\setminus \{0\}$.

The subharmonicity at $z=0$ is verified by proving sub-mean inequality:
$$\frac{1}{2r\pi}\int_{-\pi}^{\pi}H(re^{it})dt =\frac{4 \sin\left[\frac{p \pi }{2}\right]}{2r\pi p}\ge H(0)=0.$$ For a related similar statement see \cite[Remark~2.3]{verb1}.

Further, let  $r=|w|/|z|$ for $z\neq 0$  and $t=\mathrm{arg}(zw)$. Then Lemma~\ref{nice1} follows from the following lemma.

\begin{lemma}\label{lemama} For every $r\ge 0$ and $t\in[0,\pi]$ and $b>0$ we have
\begin{equation}\label{compG}G(r,b)=-r^{\frac{p}{2}} X \cos \left[\frac{1}{2} p (\pi -t)\right]-\left(1+r^2+2 r \cos t\right)^{\frac{p}{2}}+Y \left(b+r^2 \right)^{\frac{p}{2}}\ge 0\end{equation} where
$$X= 2^{2-\frac{p}{2}} \left(-\frac{\left(S_b+(1+b) \sqrt{S_b}\right) \sec\left(\frac{\pi }{p}\right)}{b}\right)^{\frac{1}{2} (-2+p)} \sin \frac{\pi}{p}$$ and
$$Y=\left(\frac{1+b+\sqrt{S_b}}{2 b}\right)^{\frac{p}{2}}.$$
The minimum is zero and it is attained  for $t=\pi-\pi/p$ and $$r=R:=-\frac{1}{2} \left(-1+b+\sqrt{S_b}\right) \sec\left(\frac{\pi }{p}\right)$$
\end{lemma}
\begin{proof}

The first step is to simplify rather complex inequality \eqref{compG} into \eqref{secondf},  which is suitable for further work.
We first have $$b=\frac{R-R^2 \cos \frac{\pi}{p}}{R-\cos \frac{\pi}{p}}.$$
Then $$X=2 \left(\frac{1}{R}+R-2 \cos \frac{\pi}{p}\right)^{\frac{1}{2} (-2+p)} \sin \frac{\pi}{p}$$ and
$$Y = \left(1-\frac{\cos \frac{\pi}{p}}{R}\right)^{\frac{p}{2}}.$$
Then our main inequality becomes
$$\left(\frac{\left(1+r^2\right) R-\left(r^2+R^2\right) \cos \frac{\pi}{p}}{r R}\right)^{\frac{p}{2}}-\left(\frac{1}{r}+r+2 \cos t\right)^{\frac{p}{2}}$$ $$-2 \left(\frac{1}{R}+R-2 \cos \frac{\pi}{p}\right)^{\frac{1}{2} (-2+p)} \cos \left[\frac{1}{2} p (\pi -t)\right] \sin \frac{\pi}{p}\ge 0$$ which need to hold for every $r,R>0$.

By making  the substitution $\frac{1}{r}+r=2\cosh \alpha$ and $\frac{1}{R}+R=2\cosh \beta$ the previous inequality becomes

$$\left({\cosh \alpha-\cosh(\alpha\pm \beta)  \cos \frac{\pi}{p}}\right)^{\frac{p}{2}}-\left(\cosh \alpha+ \cos t\right)^{\frac{p}{2}}$$ $$- \left(\cosh \beta- \cos \frac{\pi}{p}\right)^{\frac{1}{2} (-2+p)} \cos \left[\frac{1}{2} p (\pi -t)\right] \sin \frac{\pi}{p}\ge 0$$ which need to hold for every $\alpha,\beta\in \mathbb{R}$.

Then the last inequality can be written as \begin{equation}\label{secondf}\begin{split}\bigg(\cosh \alpha&-\cos \frac{\pi}{p} \cosh \gamma\bigg)^{\frac{p}{2}}-(\cos t+\cosh \alpha)^{\frac{p}{2}}\\&- \left(\cosh(\alpha-\gamma)-\cos \frac{\pi}{p}\right)^{\frac{1}{2} (-2+p)} \cos \left[\frac{1}{2} p (\pi -t)\right]\sin \frac{\pi}{p}\geq 0,\end{split}\end{equation} for $0,\gamma\le \alpha$, $t\in[0,\pi]$.
The advantage of \eqref{secondf} is that the equality case occurs precisely when $\gamma=0$ and $t=\pi-\pi/p$.

The main step of the proof of this technical Lemma is a refinement of  Terence Tao's subtle method which works for the special case when $p=3/2$ (\cite{tao}).
The first part of this method  is to replace the trigonometric expressions in $t$ with quadratic expressions in $x = -\cos \frac{\pi}{p}-\cos t$, which is a convenient further change of variable (in particular, it allows us to shift the equality case from $t=\pi-\pi/p$ to $x=0$, and eliminate $x$ shortly afterward).
\begin{lemma}\label{tt}
For $p\in[1,2]$, $t\in[0,\pi]$  and $$q=\frac{p-2 \cot\left[\frac{\pi }{2 p}\right]}{p-p \cos \frac{\pi}{p}}$$ we have
 $$\sin\left(\frac{\pi }{p}\right)\cos \left[\frac{1}{2} p (\pi -t)\right]\le \frac{1}{2} p \left(-\cos \frac{\pi}{p}-\cos t-q \left(-\cos \frac{\pi}{p}-\cos t\right)^2\right) .$$

\end{lemma}
\begin{remark}
In Lemma~\ref{tt} we are comparing two functions that coincide in two points $\pi$ and $\pi-\pi/p$. Moreover, their derivatives coincide in $\pi-\pi/p$ and this was the idea of how we came to the inequality.
\end{remark}
\begin{proof}[Proof of Lemma~\ref{tt}]
Let $a=\cos(\pi-t)=-\cos t$. Then $a\in[-1,1]$ when $t\in [0,\pi]$.
We need to show that $\Phi(a)\ge 0$ where $$\Phi(a) = \frac{1}{2} p \left(a-\cos \frac{\pi}{p}-\frac{\left(a-\cos \frac{\pi}{p}\right)^2 \left(p-2 \cot\left[\frac{\pi }{2 p}\right]\right)}{p-p \cos \frac{\pi}{p}}\right)$$ $$-\cos \left[\frac{1}{2} p \cos^{-1}(a)\right] \sin \frac{\pi}{p}.$$
Then $$\Phi'''(a) = \frac{p \sin \frac{\pi}{p} \left(6 a \sqrt{1-a^2} p \cos \left[\frac{1}{2} p t\right]+\left(-4+p^2-a^2 \left(8+p^2\right)\right) \sin\left[\frac{1}{2} p t \right]\right)}{8 \left(1-a^2\right)^{5/2}}$$ where $t=\arccos a$,  and this function is negative because

\begin{equation}\label{phip}\phi(p):=-4+p^2-\left(8+p^2\right) \cos^2 t+3 p \cot\left[\frac{p t}{2}\right] \sin(2t)\le 0.\end{equation}
Let us prove \eqref{phip}.

It is enough to prove that  $\phi$ is convex and $\phi(1)<0$ and $\phi(2)=0$. First of all $$\phi(1)=-3-9 \cos^2 t+3 \cot\left[\frac{t}{2}\right] \sin(2t)=-12 \sin\left[\frac{t}{2}\right]^4$$ and this is clearly negative for $t\in(0,\pi)$. Also,  straightforward calculations yield $\phi(2)=0$.

To prove that $\phi$ is convex we calculate $$\phi'''(p)=-\frac{3}{4} t^2\csc^4\left[\frac{p t}{2}\right]\sin(2t) (p t (2+\cos [p t])-3 \sin[p t])$$ and thus $\Phi'(p)$ is convex for $t\ge \pi/2$ and concave for $t\le \pi/2$.

Fix now $t$. If $\phi'$ is concave, because

\begin{equation}\label{prima}\phi'(1)=2 \cot\left[\frac{t}{2}\right] (-3 t \cos t+\sin t+\sin(2t))>0\end{equation} and

\begin{equation}\label{seconda}\phi'(2)=5+\cos (2t)-6 t \cot\,t>0,\end{equation} it follows that $\phi'(p)> 0$ for $p\in[1,2]$.
To prove \eqref{prima}, observe that $$\varphi(t):=-3 t \cos t+\sin t+\sin(2t)=\sum_{m=0}^\infty\frac{2 (-1)^m \left(-1+4^m-3 m\right)}{(2 m+1)!}t^{2m+1}.$$ Further two consecutive
members $m=2k$ and $m=2k+1$ of the previous sum are $$\frac{2\text{  }\left(-1+4^{2 k}-6 k\right) t^{1+4 k}}{(4 k+1)}-\frac{2\text{  }\left(-1+4^{1+2 k}-3 (1+2 k)\right) t^{3+4 k}}{(4 k+3)!},$$ and the last quantity is $\ge 0$ if and only if $$ 2 \left(-1+16^{ k}-6 k\right)-\frac{\left(-1+4 \cdot 16^{ k}-3 (1+2 k)\right) t^2}{(1+2 k) (3+4 k)}>0.$$ Since $t<4$, it is enough to prove $$\left(58+40 k-136 k^2-96 k^3\right)+\left(-58+20 k+16 k^2\right) 16^k>0.$$ Since $16^k>1 +15 k$, it will be sufficient to check  that $$-810 k+180 k^2+144 k^3>0,$$ and this is certainly true for $k\ge 2$. In particular, we arrive at the inequality

$$\varphi(t)\ge \frac{3 t^5}{20}-\frac{3 t^7}{140}+\frac{3 t^9}{2240}-\frac{t^{11}}{19800}\ge 0$$ which is true for $t\in[0,\pi]$.
Similarly, we prove \eqref{seconda}: It is clear that $\phi'(2)>0$ for $t\ge \pi/2$. For $t\le \pi/2$ we use Taylor extension  $$\sin t\phi'(2) = \sum_{m=2}^\infty \frac{3 (-1)^m \left(-1+9^m-8 m\right)}{2 (2m+1)!}t^{2m+1}$$ which is greater or equal to  $\frac{4 t^5}{5}-\frac{22 t^7}{105}\ge 0$ for $t\in[0,\pi/2]$.

 Further
 \begin{equation}\label{pa2}\phi''(1)=\csc^2\frac{t}{2}\chi(t),\end{equation} where
\begin{equation}\label{pa}\chi(t)=\left(\sin^2 t+\cos t \left(3 t^2 (1+\cos t)-\sin t (6 t+\sin t)\right)\right).\end{equation} Let us show that $\phi''(1)>0$. First,  the expression $\chi (t)$ is positive for $t\ge \pi/2$ as a sum of two positive functions. If $t\in[0,\pi/2]$, then we take the Taylor series of $\cos t$ and $\sin t$ up to some order and get the following estimates  \begin{equation}\label{est1}t-\frac{t^3}{6}\le \sin t\le t - \frac{t^3}{6} + \frac{t^5}{120}\end{equation} and \begin{equation}\label{est2}\cos t\ge 1-\frac{t^2}{2}+\left(\frac{2 \left(-8+\pi ^2\right)}{\pi ^4}\right)t^4.\end{equation}
 Let us prove the non-trivial estimate \eqref{est2}. Let $$h(t)=-1+\frac{t^2}{2}-\frac{2 \left(-8+\pi ^2\right) t^4}{\pi ^4}+\cos t,$$ and consider $g(t) =\frac{h'(t)}{t^3}.$ We have that $$g'(t)=\frac{-t (2+\cos t)+3 \sin t}{t^4}=\sum_{k=0}^\infty  \left(-\frac{2+4 k}{(4 k+5)!}+\frac{4 (1+k) t^2}{(4 k+7)!}\right) t^{1+2k}.$$ Since $$\left(-\frac{2+4 k}{(4 k+5)!}+\frac{4 (1+k) t^2}{(4 k+7)!}\right) \le\frac{1}{(4k+5)!} \left(-(2+4 k)+\frac{4 (1+k) \pi^2/4}{(7+4 k)(6+4k)}\right)<0, $$ it follows that $g$ is decreasing. Since $g(0)=\frac{1}{6}+\frac{64}{\pi ^4}-\frac{8}{\pi ^2}>0$ and $g(\pi/2)=-\frac{4 \left(-16+2 \pi +\pi ^2\right)}{\pi ^4}<0$. There is only one point $t_\circ\in(0,\pi/2)$ so that $g(t_\circ)=h'(t_\circ)=0$. Since $h(\frac{\pi}{4})=-\frac{15}{16}+\frac{1}{\sqrt{2}}+\frac{3 \pi ^2}{128}>0$. It follows that $t_\circ$ is a local maximum of $h$. Since $h(0)=h(\pi/2)=0$, this implies that $h(t)\ge 0$ for $t\in[0,\pi/2]$.

 Then \eqref{pa}, \eqref{est1} and \eqref{est2} imply  that  $$\chi(t)\ge t^6\zeta(t)$$ where
 \[\begin{split}\zeta(t)&=\left(\frac{1}{60}-\frac{32}{\pi ^4}+\frac{4}{\pi ^2}\right) +\left(\frac{1}{20}+\frac{80}{3 \pi ^4}-\frac{10}{3 \pi ^2}\right) t^2\\&+\left(-\frac{7}{4800}+\frac{768}{\pi ^8}-\frac{192}{\pi ^6}+\frac{608}{45 \pi ^4}-\frac{17}{90 \pi ^2}\right) t^4\\&+\frac{\left(-1280+160 \pi ^2+\pi ^4\right) t^6}{28800 \pi ^4}+\frac{\left(8-\pi ^2\right) t^8}{7200 \pi ^4}.\end{split}\] Then
 $$\zeta(t)\approx \left(0.09344 -0.0139778 t^2-0.000663 t^4+0.000141 t^6-0.0000026657 t^8\right)$$ which is apparently positive for $t\in[0,\pi/2]$.

Since $\phi'$ is convex, for $p>1$ we obtain that  $\phi''(p)>\phi''(1)>0$. Thus  $\phi$ is convex and hence $\phi(p)<0$ because  $\phi(1)<0$ and $\phi(2)=0$.

 The conclusion is that $$\Phi'(a)=\frac{1}{2} p \left(1+\frac{2 \left(a-\cos \frac{\pi}{p}\right) \left(p-2 \cot\left[\frac{\pi }{2 p}\right]\right)}{p \left(-1+\cos \frac{\pi}{p}\right)}-\frac{\sin \frac{\pi}{p} \sin\left[\frac{1}{2} p \cos^{-1}(a)\right]}{\sqrt{1-a^2}}\right)$$ is concave. Since $\Phi'(\cos(\pi/p))=0$, it follows that there is at most one more point $-1<\bar a <1$ so that $\Phi'(\bar a)=0$. We now have \begin{equation}\label{xp}\Phi''(\cos(\frac{\pi}{p}))=\frac{1}{4} \csc^3\left(\frac{\pi }{p}\right) \left(32 \cos^4\left[\frac{\pi }{2 p}\right]-4 p \sin \frac{\pi}{p}-3 p \sin\left[\frac{2 \pi }{p}\right]\right)>0.\end{equation}  In order to prove \eqref{xp} observe that, after taking the change $p=\pi/x$, it is equivalent with $$\left(32 x \cos^4\left[\frac{x}{2}\right]-4 \pi  \sin x-3 \pi  \sin(2x)\right)\ge 0, \ \  x\in[\pi/2,\pi].$$ The last relation can be written as $$8 x (1+\cos x)^2-2 \pi  (2+3 \cos x) \sin x\ge 0.$$

After the change $x=y+\pi/2$ we arrive at the inequality $$h(y):=-2 \pi  \cos y(2-3 \sin y)+8 \left(\frac{\pi }{2}+y\right) (1-\sin y)^2\ge 0.$$ Now it is clear that for $y\in[\arcsin \frac{2}{3},\pi/2]$, $h(y)\ge 0$. For $y\in[0,\arcsin \frac{2}{3}]$, we use the inequality $$h(y)\ge k(y):=-2 \pi  (2-3 \sin y)+8 (1-\sin y)^2 \left(\frac{\pi }{2}+\sin y\right)=h(\sin y),$$ where $h(t)=-2 \pi  (2-3 t)+8 (1-t)^2 \left(\frac{\pi }{2}+t\right)$. Since $h(t)\ge 0$ for $t\in[0,2/3]$,  \eqref{xp} is proved.

Since $\Phi''$ is decreasing, it follows that $\bar a>\cos \pi/p$.  Then $\Phi$ is decreasing in $[-1,\cos(\pi/p)]$, increasing in $[\cos(\pi/p),\bar a]$ and decreasing in $[\bar a, 1]$. Since $\Phi(1)=0$ and $\Phi(\cos(\pi/p))=0$, it follows that $\Phi(a)\ge 0$ for every $a$. This finishes the proof of the lemma.
\end{proof}
{\bf Continuation of proof of Lemma~\ref{nice1}}
We begin with the required inequality
$$\left(\cosh \alpha-\cos \frac{\pi}{p} \cosh \gamma\right)^{\frac{p}{2}}-(\cos t+\cosh \alpha)^{\frac{p}{2}}$$ $$- \left(\cosh(\alpha-\gamma)-\cos \frac{\pi}{p}\right)^{\frac{1}{2} (-2+p)} \cos \left[\frac{1}{2} p (\pi -t)\right]\sin \frac{\pi}{p}\geq 0,$$ for $0,\gamma\le \alpha$.
Let $x=-(\cos t + \cos \frac{\pi}{p})$.
From Lemma~\ref{tt}, it is enough to prove \[\begin{split}&\left(\cosh\alpha -\cos \frac{\pi}{p} \cosh\gamma  \right)^{\frac{p}{2}}-\left(-x-\cos \frac{\pi}{p}+\cosh\alpha \right)^{\frac{p}{2}}\\&\ge \left(-\cos \frac{\pi}{p}+\cosh [\alpha -\gamma ]\right)^{\frac{p}{2}-1} \left(\frac{p x}{2}-\frac{x^2 \left(p-p \cos \frac{\pi}{p}-2 \sin \frac{\pi}{p}\right)}{2 \left(1-\cos \frac{\pi}{p}\right)^2}\right)\end{split}\]
where $x=-\cos\frac{\pi }{p}-\cos t\in[-\cos\frac{\pi }{p}-1,-\cos\frac{\pi }{p}+1]$.
\\
From concavity of $s\to s^{\frac{p}{2}}$ we have $$-\left(-x-\cos \frac{\pi}{p}+\cosh\alpha \right)^{\frac{p}{2}}\ge -\left(\cosh\alpha-\cos \frac{\pi}{p} \right)^{\frac{p}{2}}+\frac{p}{2} \left(\cosh\alpha-\cos \frac{\pi}{p} \right)^{\frac{p}{2}-1}x.$$
It is enough to prove that \[\begin{split}\frac{1}{2} p x &\left(\cosh\alpha-\cos \frac{\pi}{p} \right)^{\frac{p}{2}-1}-\left(\cosh\alpha-\cos \frac{\pi}{p} \right)^{\frac{p}{2}}\\&+\left(\cosh\alpha -\cos \frac{\pi}{p} \cosh\gamma  \right)^{\frac{p}{2}}\\&\geq \left(-\cos \frac{\pi}{p}+\cosh [\alpha -\gamma ]\right)^{\frac{p}{2}-1} \left(-\frac{p x}{2}-\frac{x^2 \left(p-p \cos \frac{\pi}{p}-2 \sin \frac{\pi}{p}\right)}{2 \left(1-\cos \frac{\pi}{p}\right)^2}\right) .\end{split}\]
Further $$\left(\cosh\alpha -\cos \frac{\pi}{p} \cosh\gamma  \right)^{\frac{p}{2}}-\left(\cosh\alpha-\cos \frac{\pi}{p} \right)^{\frac{p}{2}}$$ $$>\frac{1}{2} p \cos \frac{\pi}{p} \left(\cosh\alpha-\cos \frac{\pi}{p} \right)^{\frac{p}{2}-1} (1-\cosh\gamma  ).$$
Thus we need to show that $Wx^2+V x+U>0$ where $$W=\frac{\left(\cosh(\alpha-\gamma)-\cos \frac{\pi}{p}\right)^{\frac{p}{2}-1} \left(p-p \cos \frac{\pi}{p}-2 \sin \frac{\pi}{p}\right)}{2 \left(1-\cos \frac{\pi}{p}\right)^2},$$
$$V=\frac{1}{2} p \left(\cosh \alpha-\cos \frac{\pi}{p}\right)^{\frac{p}{2}-1}-\frac{1}{2} p \left(\cosh(\alpha-\gamma)-\cos \frac{\pi}{p}\right)^{\frac{p}{2}-1}$$
and $$U=\frac{p \cos \frac{\pi}{p} \left(\cosh \alpha-\cos \frac{\pi}{p}\right)^{\frac{p}{2}} (-1+\cosh \gamma)}{2 \left(\cos \frac{\pi}{p}-\cosh \alpha\right)}.$$
Since $W>0$, it is clear that it is enough to prove that the  discriminant is negative. Now $$V^2- 4 U W<0 $$ if and only if \[\begin{split}G(X)&:=\frac{1}{4} p \left(1-\left(\frac{1}{X}\right)^{1-\frac{p}{2}}\right)^2 X^{1-\frac{p}{2}}\\&+\frac{\cos \frac{\pi}{p} (-1+\cosh \gamma) \left(p-p \cos \frac{\pi}{p}-2 \sin \frac{\pi}{p}\right)}{\left(1-\cos \frac{\pi}{p}\right)^2}\le 0\end{split}\] where $$X=X(\alpha)=\frac{-\cos \frac{\pi}{p}+\cosh \alpha}{-\cos \frac{\pi}{p}+\cosh(\alpha-\gamma)}.$$
It turns out that we need to split into two cases $\gamma\ge 0$ and $\gamma<0$.

\subsection{\bf The case $\gamma\ge 0$} First of all

$$G'(X)=\frac{1}{8} (2-p) p X^{-2-\frac{p}{2}} \left(X^2-X^p\right)\ge 0$$ for $ X\ge 1$. Further $$X'(\alpha)=\frac{\cos \frac{\pi}{p} (-\sinh \alpha+\sinh(\alpha-\gamma))+\sinh \gamma}{\left(\cos \frac{\pi}{p}-\cosh(\alpha-\gamma)\right)^2}\ge 0$$ and so $$G(X)\le G(\lim_{\alpha\to \infty } X(\alpha))=G(e^\gamma).$$
Let $\gamma=\log x$. It remains to prove that \begin{equation}\label{goal}\frac{1}{4} \left(\frac{p x^{-1-\frac{p}{2}} \left(x-x^{\frac{p}{2}}\right)^2}{-2+\frac{1}{x}+x}+\frac{2 \cos \frac{\pi}{p} \left(p-p \cos \frac{\pi}{p}-2 \sin \frac{\pi}{p}\right)}{\left(1-\cos \frac{\pi}{p}\right)^2}\right)\le 0\end{equation} for $x\ge 1$ and $p\in[1,2]$. Now we prove that  \begin{equation}\label{expression}\sup_{x>0,x\neq 1}\frac{x^{-1-\frac{p}{2}} \left(x-x^{\frac{p}{2}}\right)^2}{-2+\frac{1}{x}+x}=\lim_{x\to 1}\left(\frac{x^{1-p/4}-x^{p/4}}{x-1}\right)^2=\frac{1}{4} (2-p)^2.\end{equation}

To prove the first equality in \eqref{expression}, assume first that $x>1$. Then we need to show that $$\phi(x):=-\left(1-\frac{p}{2}\right) (-1+x)+x^{1-\frac{p}{4}}-x^{p/4}\le 0.$$ Since $$\phi''(x)=\frac{1}{16} (4-p) p x^{-2-\frac{p}{4}} \left(-x+x^{\frac{p}{2}}\right)\le 0,$$ it follows that $\Phi'$ is decreasing. In particular, $\phi'(x)<\phi'(1)=0$ for $x>1$. Thus $\phi$ is decreasing and hence $\phi(x)\le \phi(1)=0$, what we wanted to prove. The case $x<1$ can be treated similarly.

So it remains to show that $$\frac{1}{4} (-2+p)^2+\frac{2 \cos \frac{\pi}{p} \left(p-p \cos \frac{\pi}{p}-2 \sin \frac{\pi}{p}\right)}{\left(1-\cos \frac{\pi}{p}\right)^2}\le 0$$ for $p\in[1,2]$. This inequality is equivalent to the inequality

\begin{equation}\label{sabato}s^2 (1+\sin s)^2-(\pi +2 s) \sin s \left[\pi -(\pi +2 s) \cos s+\pi  \sin s\right]\le 0\end{equation} for $s=\frac{\pi}{p}-\frac{\pi}{2}\in[0,\pi/2]$.

We first have that \begin{equation}\label{18sept}\psi(s):=\pi -(\pi +2 s) \cos s+\pi  \sin s\ge ( (-2+\pi ) s+\frac{\pi  s^2}{2}).\end{equation}
To prove \eqref{18sept}, since $$\sin s\ge s-\frac{s^3}{6}\text{ and }\cos s\le 1-\frac{s^2}{2}+\frac{s^4}{24},$$ we have that
\[\begin{split}
\psi(s)&\ge \pi -(-2+\pi ) s-\frac{\pi  s^2}{2}+\pi  \left(s-\frac{s^3}{6}\right)-(\pi +2 s) \left(1-\frac{s^2}{2}+\frac{s^4}{24}\right)
\\& =\frac{1}{24} s^3 \left(24-2 s^2-\pi  (4+s)\right)\ge0, \ \ s\in[0,\pi/2].\end{split}\] This implies \eqref{18sept}.
So to prove \eqref{sabato}, it is enough to prove the inequality $$s^2 (1+\sin s)^2-(\pi +2 s) \sin s ((-2+\pi ) s+\frac{\pi  s^2}{2})\le 0.$$ Since $$t-\frac{t^3}{6}\le \sin t\le t,$$ it remains to prove the inequality $$\omega(s):=\left(1+2 \pi -\pi ^2\right) +\left(6-2 \pi -\frac{\pi ^2}{2}\right) s+\frac{1}{6} \left(6-8 \pi +\pi ^2\right) s^2$$ $$+\frac{1}{12} \left(-8+4 \pi +\pi ^2\right) s^3+\frac{\pi  s^4}{6}\le 0. $$

We have that $$\omega'''(s)=\frac{1}{2} \left(-8+4 \pi +\pi ^2\right)+4 \pi  s>0.$$ So $\omega'(s)$ is convex. In particular, $\omega'(s)$ can have at most two zeros. Since $\omega'(0)<0$ and $\omega'(\pi/2)>0$. There is only one point $s_\circ \in(0,\pi/2)$ so that $\omega'(s_\circ)=0$ and $s_\circ$ is a local minimum of $\omega$. Since $\omega(0)<0$ and $\omega(\pi/2)<0$, it follows that $\omega(s)<0$ for $s\in[0,\pi/2]$.
This finishes \eqref{goal} and this case is finished.

\subsection{\bf Case $\gamma< 0$} In this case $G'(X)\le 0$ and $X'(a)\ge 0$. So $$G(X(\alpha))\le G(X(0))=G\left(\frac{1-\cos \frac{\pi}{p}}{\cosh \gamma-\cos \frac{\pi}{p} }\right)\le G(e^\gamma),$$ because $$\frac{1-\cos \frac{\pi}{p}}{{\cosh \gamma}-\cos \frac{\pi}{p}}\ge e^\gamma,$$ for $\gamma\le 0$.  In this case $\gamma<0$, but we use again \eqref{expression}, and the inequality is proved.
\end{proof}
%\subsection*{Acknowledgments} I would like to thank professor Terence Tao for emailing the proof of particular case of Lemma~\ref{lemama} for $p=3/2$.


\begin{thebibliography}{1}
\bibitem{essen}
\textsc{M. Ess\'en,} \emph{A superharmonic proof of the M. Riesz conjugate function theorem.} Ark. Mat. 22 (1984), no. 2, 241--249.
%\bibitem{hh}
%\textsc{H. Hedenmalm,} Bloch functions, asymptotic variance, and geometric zero packing,  arXiv:1602.03358.

\bibitem{studia}
\textsc{ B. Hollenbeck,  N. J. Kalton, I. E. Verbitsky, }
\emph{Best constants for some operators associated with the Fourier and Hilbert transforms. }
Studia Math. 157 (2003), no. 3, 237--278.
\bibitem{verb1}
\textsc{B. Hollenbeck, I. E. Verbitsky,} \emph{Best constants for the Riesz projection.} J. Funct. Anal. 175 (2000), no. 2, 370--392.
\bibitem{verb2}
\textsc{B. Hollenbeck, I. E. Verbitsky,} \emph{
Best constant inequalities involving the analytic and co-analytic projection.} Operator Theory: Advances and Applications 202, 285-295 (2010).
\bibitem{lars}
\textsc{L. H\"ormander,} \emph{An introduction to complex analysis in several variables.} D. Van Nostrand Co., Inc., Princeton, N.J.-Toronto, Ont.-London 1966 x+208 pp.
\bibitem{graf}
\textsc{L. Grafakos,}
\emph{Best bounds for the Hilbert transform on $L^p(\mathbf{R})$.}
Math. Res. Lett. 4 (1997), no. 4, 469--471.
\bibitem{tams}
\textsc{D.  Kalaj,} \emph{On Riesz type inequalities for harmonic mappings on the unit disk,} Trans. Am. Math. Soc.
 372, No. 6, 4031--4051 (2019).
%\bibitem{ka}
%\textsc{D.  Kalaj,}  \emph{Isoperimetric inequality for the polydisk.} Ann. Mat. Pura Appl. (4) 190 (2011), no. 2, 355--369
\bibitem{mel} \textsc{P. Melentijevi\'c,} \emph{Hollenbeck-Verbitsky conjecture on best constant inequalities for analytic and co-analytic projections.} Math. Ann. (2023). https://doi.org/10.1007/s00208-023-02639-1
\bibitem{melmar}
\textsc{P. Melentijevi\'c, M. Markovi\'c,} \emph{Best Constants in Inequalities Involving Analytic and Co-Analytic Projections and Riesz's Theorem in Various Function Spaces.} Potential Anal (2022). https://doi.org/10.1007/s11118-022-10021-0
\bibitem{pik}
\textsc{S. K. Pichorides, S.}
\emph{On the best values of the constants in the theorems of M. Riesz, Zygmund and Kolmogorov.}
Collection of articles honoring the completion by Antoni Zygmund of 50 years of scientific activity, II.
Studia Math. 44 (1972), 165--179.
\bibitem{tao}
\textsc{T. Tao,} \emph{https://mathoverflow.net/questions/454007}
\bibitem{ver}
\textsc{I. E. Verbitsky,}
\emph{Estimate of the norm of a function in a Hardy space in terms of the norms of its
real and imaginary parts.}
Linear operators.
Mat. Issled. No. 54 (1980), 16--20, 16--165, "Fifteen Papers on Functional Analysis," Amer. Math. Soci.
Transl. Ser. 124 (1984), 11--15 (English transl.)
%\bibitem{ZY}
%\textsc{A. Zygmund,} \emph{ Trigonometric series. Vol. I, II. Reprinting of the 1968 version of the second edition with Volumes I and II bound together. } Cambridge University Press, Cambridge-New York-Melbourne, 1977. Vol. I: xiv+383 pp.; Vol. II: vii+364 pp.
\end{thebibliography}
\end{document}